\documentclass{amsart}
\usepackage{amssymb} 
\usepackage{amsmath} 
\usepackage{amscd}
\usepackage{amsbsy}
\usepackage{comment, enumerate}
\usepackage[matrix,arrow]{xy}
\usepackage[hidelinks]{hyperref}
\usepackage{mathrsfs}
\usepackage{color}
\usepackage{mathtools,caption}
\usepackage{longtable}
\DeclareMathOperator{\ord}{ord}

\DeclareMathOperator{\norm}{Norm}

\newcommand{\QQ}{\mathbb Q}
\newcommand{\F}{\mathbb F}
\newcommand{\ZZ}{\mathbb Z}
\newcommand{\OK}{\mathcal O_K}
\newcommand{\fa}{\mathfrak a}

\newcommand{\fp}{\mathfrak p}
\newcommand{\fq}{\mathfrak q}

\newcommand{\tu}{\tilde u}
\newcommand{\bound}{10^4}

\newcommand{\fz}{\mathfrak z}

\newcommand{\OL}{\mathcal O_L}

\newcommand\margnote[1]{\marginpar{\tiny\begin{minipage}{20mm}\begin{flushleft} #1\end{flushleft}\end{minipage}}}

\newtheorem{Th}{Theorem}[section]
\newtheorem{Lemma}[Th]{Lemma}

\theoremstyle{definition}

\theoremstyle{remark}
\newtheorem{Remark}[Th]{Remark}

\definecolor{cadmiumgreen}{rgb}{0.0, 0.42, 0.24}

\begin{document}

\title[]{Perfect powers that are sums of squares of an arithmetic progression}

\author{Debanjana Kundu}
\address{Department of Mathematics, University of Toronto\\
Bahen Centre, 40 St. George St., Room 6290, Toronto, Ontario, Canada, M5S 2E4}
\email{dkundu@math.utoronto.ca}

\author{Vandita Patel}
\address{Department of Mathematics, University of Toronto, Bahen Centre, 40 St. George St., Room 6290, Toronto, Ontario, Canada, M5S 2E4}
\email{vandita@math.utoronto.ca}

\date{\today}

\keywords{Exponential equation, Lehmer sequences, primitive divisors}
\subjclass[2010]{Primary 11D61}

\begin{abstract}
In this paper, we determine all primitive solutions
to the equation $(x+r)^2  + (x + 2r)^2 + \cdots + (x+dr)^2 = y^n$ for $2\leq d\leq 10$  and for $1 \le r \le \bound$ 
%(except in the case $d=6$, where we restrict $1\leq r \leq 5000$). 
We make use of a factorization argument and the Primitive Divisors Theorem 
due to Bilu, Hanrot and Voutier.
\end{abstract}

\maketitle

%------------------------------------------
\section{Introduction} \label{intro}
%------------------------------------------

Finding perfect powers that are sums of terms in an arithmetic progression has received much interest; recent contributions can be found in 
\cite{ArgaezPatel}, \cite{BGP}, \cite{BPS1}, \cite{BPS2}, \cite{BPSSoydan}, 
%%%\cite{Br}, \cite{Cassels}, \cite{GTV}, 
\cite{Ha}, 
\cite{Patel18}, \cite{PatelSiksek}, 
%%%\cite{P}, \cite{P2}, \cite{Sc}, 
\cite{Soydan}, \cite{Zhang}, \cite{ZZhang} and \cite{ZB}.
In this paper, we consider the equation:
\begin{equation}\label{eq:main}
(x+r)^2  + (x + 2r)^2 + \cdots + (x+dr)^2 = y^n \quad x,y,r,n \in \ZZ,\; n \ge 2.
\end{equation}
%equation~{eq:AP} for $k=2$, and with $2 \leq d \leq 10$. 
We say an integer solution $(x,y)$ of equation~\eqref{eq:main} is \textit{primitive} 
if $\gcd(x,y,r)=1$. 
%This is equivalent to $x,y,r$ being pairwise coprime. If $(x,y)$ is primitive, then $xy\neq 0$.

In this paper, we prove the following theorem:

%We completely solve \eqref{eq:main} for $2 \leq d \leq 10$ and for all $1\leq r\leq \bound$ (except in the case $d=6$, where restrict $1\leq r \leq 5000$), under the
%natural assumption $\gcd(x,y)=1$. Our main approach uses the characterization of primitive
%divisors in Lehmer sequences due to Bilu, Hanrot and Voutier \cite{BHV}.

\begin{Th}\label{thm:main}
Let $2\leq d\leq 10$ and $1\leq r \leq 10^4$. 
All primitive solutions to equation~\eqref{eq:main}  for $d=2$ with prime exponent $n\geq3$ are given in Table~\ref{table:solutionsAP2}, and with exponent $n=4$ are given in Table~\ref{table:solutionsAP24}. For $d=3$, all primitive solutions are given in \cite{KoutsianasPatel} and for $d=6$, all primitive solutions are recorded in Table~\ref{table:solutionsAP6}.

When $d=2$ and $n=2$, we have no solutions unless every prime divisor
of $r$ is congruent to $\pm 1 \pmod{8}$. Suppose 
we are in the latter case, let $r=q_1^{t_1} \cdots q_s^{t_s}$
where the $q_i$ are distinct primes. For each $i$ we may
write $q_i=\norm(\mathfrak{q}_i)$ with $\mathfrak{q}_i \in \ZZ[\sqrt{2}]$.
Then the solutions are given by
\[
2x+3r+y\sqrt{2}=\pm \mathfrak{r}^2 \cdot (1+\sqrt{2})^{2k+1}, \qquad k \in \ZZ
\]
where $\mathfrak{r}=\mathfrak{r}_1^{t_1} \cdots \mathfrak{r}_s^{t_s}$
and each $\mathfrak{r}_i$ is either $\mathfrak{q}_i$ or its
conjugate $\overline{\mathfrak{q}_i}$.
\end{Th}

In this paper, we have incorporated the results of \cite{Cohn}
into Theorem~\ref{thm:main} %and \ref{thm:main2} 
where equation~\eqref{eq:main} was solved for $d=2,r=1, n\geq 3$.  
In the proof of Theorem~\ref{thm:main}, we are able to solve the case $d=2,r=1, n\geq 3$ uniformly with other values of $r\leq \bound$, using a different approach to that taken in \cite{Cohn}.
We have also consolidated the results of 
\cite{KoutsianasPatel} into Theorem~\ref{thm:main}, where the second author and Koustsianas list all primitive solutions to equation~\eqref{eq:main} for the case $d=3$ and prime exponent $n$. 
General theorems on equations of the form $x^2 + C = 2y^p$ can be found in  \cite{AMLST}, \cite{Tengely} and  \cite{Tengely2}. 
As in \cite{AMLST}, \cite{KoutsianasPatel} and \cite{Patel18}, our main tool is the characterization of primitive divisors in Lehmer sequences due to Bilu, Hanrot and Voutier \cite{BHV}.

\begin{Remark}

\begin{enumerate}

\item Primitive solutions to equation~\eqref{eq:main} with exponent $n$ that is composite can be recovered from Tables~\ref{table:solutionsAP2}, \ref{table:solutionsAP24} and \ref{table:solutionsAP6} by checking whether $y$ is a perfect power.

%\item For $d=6$, we could only produce results for $1 \leq r \leq 5000$ due to a computational limitation. This arises due to the presence of extremely large coefficients of certain polynomials.

\item Values of $d$ where all prime divisors of $d$ are congruent to $\pm 1 \pmod{12}$ are not amenable to the techniques used in this paper, hence we have the restriction $2 \leq d \leq 10$. 
\end{enumerate}

\end{Remark}

We now briefly explain the organization of the paper. In $\S \ref{sec: precursory lemmata}$ we record some preliminary lemmata. In $\S \ref{sec:n=2}$ and $\S \ref{sec:n=4}$, we solve equation~\eqref{eq:main} when $n$ is even for relevant $d$. In $\S \ref{section: Primitive prime divisors of Lucas and Lehmer sequences}$, we recall the main theorem of \cite{BHV} which is essential in proving Theorem~\ref{thm:main}. %and \ref{thm:main6}. %used to prove our results in the following two sections. 
In $\S \ref{sec:2AP}$ and $\S \ref{sec:6AP}$, we solve equation~\eqref{eq:main} for the case $d=2$ and $d=6$ respectively, with $n$ an odd prime. This complete the proof of the theorem stated in $\S 1$. All tables of solutions can be located in $\S \ref{sec:solntables}$.

%
%------------------------------------------
\section*{Acknowledgement}
%------------------------------------------
%
The second author is incredibly indebted to the University of Toronto and their mathematics department for their amazing support during such a difficult period.
In this regard, both authors are most thankful to Facebook Messenger, which greatly supported mathematics communications across continents. We also thank Professor Szabolcs  Tengely for his comments on a first draft of this manuscript and for pointing towards work that helped us to amend some proofs and greatly increase the efficiency and speed of computations.
Special thanks go to Professor Kumar Murty and to GANITA Lab as they are all full of sheer awesomeness.
The second author would also like to thank Celine, Heline and Dan, Jihad and Jeanne,  Lilit, Milena and Tim, Mirna and Simon, Pete, Priti, Priya and Sapna for their continuing support.
Last, but definitely not least, the second author wants to extend all of her gratitude and love towards Pravin, whose support this past year has been crucial to all of her mathematical endeavours.

%
%Both authors are grateful to the referee for a careful reading of this paper and for their helpful comments and corrections.
%

%
%------------------------------------------
\section{Some precursory lemmata} {\label{sec: precursory lemmata}}
%------------------------------------------
%
In this section, we adapt Lemma 2.2 and 2.3 from \cite{Patel18} and Lemma 2.3 from \cite{PatelSiksek}.
We rewrite equation~\eqref{eq:main}  as
\begin{equation}\label{eqn:conssquares1}
dx^2 + d(d+1)xr + \frac{d(d+1)(2d+1)}{6}r^2 = y^n.
\end{equation}

Factorising and completing the square gives us
\begin{equation}\label{eqn:conssquares}
d \left(\left( x + \frac{d+1}{2}r \right)^2  + \frac{(d-1)(d+1)}{12}r^2 \right) = y^n.
\end{equation}

Observe in particular that $y \ne 0$.

\begin{Lemma}\label{lem:1}
Let $j = \ord_2(d)$. If $j\geq 2$, then in equation~\eqref{eqn:conssquares} we have $n \mid (j-1)$.
\end{Lemma}
\begin{proof}
Let $D = d/2^2$. We substitute into equation~\eqref{eqn:conssquares} to get,
\begin{equation}\label{eqn:dval}
D \left(\left( 2x + (d+1)r \right)^2  + \frac{(d-1)(d+1)}{3}r^2 \right) = y^n.
\end{equation}
Since $j \geq 2$, equation~\eqref{eqn:conssquares1} shows that $2 \mid y$ and therefore, $2 \nmid r$ since $(x,y)$ is a primitive solution.
Observe that 
\[
(2x+(d+1)r)^2 \equiv 1 \pmod{4}, \qquad
\frac{(d-1)(d+1)}{3}r^2 \equiv 1 \pmod{4}.
\]
Comparing valuations on both sides of equation~\eqref{eqn:dval} we see that
\[
n \ord_2(y)=\ord_2(D)+1=j-1.
\]
This completes the proof.
\end{proof}

\begin{Lemma}\label{lem:2}
Let $j = \ord_3(d)$. If $j\geq 2$,  then in equation~\eqref{eqn:conssquares}, we have $n \mid (j-1)$.
\end{Lemma}
\begin{proof}
Let $D = d/3$. We substitute into equation~\eqref{eqn:conssquares} to get,
\[
D \left(3\left( x + \frac{(d+1)}{2}r \right)^2  + \frac{(d-1)(d+1)}{4}r^2 \right) = y^n.
\]
Since $j \geq 2$, equation~\eqref{eqn:conssquares1} asserts that $3 \mid y$ and therefore $3 \nmid r$ since we assume throughout that $\gcd(y,r) = 1$. 
%under the pairwise co-primality condition on $x,y$ and $r$.
Observe that the expression in brackets is never divisible by 3. Hence 
$\ord_3(D)=\ord_3(y^n)=n \ord_3(y)$, thus proving the lemma.
%that $n \mid (j-1)$.
\end{proof}

\begin{Lemma}\label{lem:5mod12}
Let $r$ be a non-zero positive integer. Let $q$ 
%q \geq 5$ 
be a prime 
%{\color{red}
%not dividing $r$ 
%}
such that 
%the congruence $B_2(x) \equiv 0 \pmod{q}$ has no solutions. 
$q \equiv \pm 5 \pmod{12}$.
Let $d$ be a positive integer such that %$q \mid d$ and 
$\ord_q(d) \not\equiv 0 \pmod{n}$. Then equation~\eqref{eq:main} has no solutions.
\end{Lemma}

\begin{proof}
Our assumption on $q$ forces $q \ne 2,3$, 
and equation~\eqref{eqn:conssquares1} affirms that $q \mid y$. 
%This is because, we require $q \geq 5$, which is the case $k=2$ in the notation of \cite{PatelSiksek}). If $q\mid r$ then we see that $q \mid x$, contradicting our pairwise co-primality assumption on $r,y$. 
Since $(x,y)$ is primitive, $q \nmid r$. As $d \equiv 0 \pmod{q}$
and $\ord_{q}(d) \not\equiv 0 \pmod{n}$, equation \eqref{eqn:conssquares}
tells us that
\[
\left(x+\frac{r}{2}\right)^2 \equiv \frac{1}{12} \pmod{q}.
\]
This implies $q \equiv \pm 1 \pmod{12}$ which gives a contradiction.
%Using \cite{PatelSiksek}[Lemma 2.3], we are only left to check that $B_2(x) \equiv 0 \pmod{q}$ has no solutions.
%This is equivalent to $3$ not being a square modulo $q$ which then yields the required condition on $q$.
%only happens when $q \equiv \pm 5 \pmod{12}$.
\end{proof}

Applying Lemmata~\ref{lem:1},~\ref{lem:2} and ~\ref{lem:5mod12} allows us to eliminate $d=4,5,7,9,10$ for all $n\geq 2$,
and $d=8$  with $n \ge 3$. For the proof of Theorem~\ref{thm:main}, it remains to deal
with $d=2,3$ and $6$ for $n\geq 2$, and also with $d=8$ for $n=2$.
The case $d=3$ and $ r\leq \bound$ has been resolved in \cite{KoutsianasPatel} and a table of solutions can be found in that paper.
%
%------------------------------------------
%------------------------------------------
\section{Case: $n=2$}\label{sec:n=2}
%------------------------------------------
%------------------------------------------

In this section, we deal with the case $n=2$ when $d=2,6,8$. 
\begin{Lemma}
Let $d=6$ or $8$ and $n=2$. Then equation~\eqref{eq:main} has no integer solutions.
\end{Lemma}

\begin{proof}
Let $d=6$, $n=2$. We rewrite equation~\eqref{eq:main} as
\[
3(2x+7r)^2 + 35r^2 = 2y^2.
\]
As $6$ is a non-square modulo $7$, we see that $7 \mid (2x+7r)$ and $7 \mid y$
which quickly contradicts primitivity.
%Note that $2 \nmid r$ otherwise we contradict $\gcd(x,r)=1$. The above equation has no  integer solutions as the left-hand side is congruent to $6\pmod{8}$ and the right-hand side is congruent to $0$ or $2 \pmod{8}$ according to whether $y$ is even or odd. 

When $d=8$, $n=2$, we rewrite equation~\eqref{eq:main} as
\[
2((2x+9r)^2 + 21r^2) = y^2.
\]
Writing $y = 2Y$ we obtain
\[
(2x+9r)^2 + 21r^2 = 2Y^2
\]
and considering the equation modulo $3$, we see that $2$ must be a square modulo $3$ and arrive at a contradiction.
\end{proof}

It remains to deal with $d=2$, $n=2$. Here we prove the claim made
about this case in the statement of Theorem~\ref{thm:main}.

\begin{proof}[Proof of Theorem~\ref{thm:main} for $d=2$, $n=2$]
In this case \eqref{eq:main} is
\[
2 x^2+6xr+5r^2=y^2.
\]
We see from this that $r$ is odd (otherwise the solution is imprimitive).
We may rewrite this as
\begin{equation}\label{eq:n2d2}
(2x+3r)^2-2y^2=-r^2.
\end{equation}
It follows that $2$ is a quadratic residue modulo any prime
divisor of $r$ and so they are all of the form $\pm 1 \mod{8}$,
and so split in $\ZZ[\sqrt{2}]$. Write
$r=q_1^{t_1} \cdots q_s^{t_s}$ as in the theorem where
the $q_i$ are distinct primes. For each $i$
let $\mathfrak{q}_i \in \ZZ[\sqrt{2}]$ satisfy
$\norm(\mathfrak{q}_i)=q$; these $\mathfrak{q}_i$
are primes of $\ZZ[\sqrt{2}]$.
Thus the prime divisors of $2x+3r+y\sqrt{2}$, $2x+3r-y\sqrt{2}$
are among $\mathfrak{q}_i$, $\overline{\mathfrak{q}_i}$.
As the solution is primitive 
$2x+3r+y\sqrt{2}$, $2x+3r-y\sqrt{2}$ are coprime in $\ZZ[\sqrt{2}]$.
Thus $2x+3r+y\sqrt{2}$ is divisible by either $\mathfrak{q}_i$
or $\overline{\mathfrak{q}_i}$ but not both, and moreover, the valuation
at this prime is $2 t_i$. Thus $2x+3r+y\sqrt{2}=\epsilon \cdot \mathfrak{r}^2$
where $\epsilon$ is a unit, and $\mathfrak{r}$ is as in the statement of 
the theorem. Taking norms and comparing to \eqref{eq:n2d2}
we see that $\norm(\epsilon)=-1$, so $\epsilon=\pm (1+\sqrt{2})^{2k+1}$
for some integer $k$. This completes the proof in this case.
\end{proof}

\begin{Remark}
%For the case where 
As we see infinitely many solutions arising in the case $d=2, n=2$, 
%we 
%also need to check for solutions to 
%must also 
it remains to solve equation~\eqref{eq:main} with $n=4$. This will be done in the next section.
\end{Remark}

%------------------------------------------
%------------------------------------------
\section{Case: $n=4$}\label{sec:n=4}
%------------------------------------------
%------------------------------------------
%
%Let $n=4$. 
In this section, we find all integer solutions to the equation:
\[
(x+r)^2 + (x+2r)^2 = y^4.
\]

We note that since $\gcd(x,r) =1$ we must have $\gcd(x+r,x+2r)=1$.
We denote $\sqrt{-1} = i$.
Applying a descent argument over the Gaussian integers, we obtain:
\begin{equation}\label{eqn:desci}
(x+2r) + i(x+r) = \epsilon \alpha^4
\end{equation}
where $\epsilon \in \{\pm1, \pm i\}$ is a unit and $\alpha \in \ZZ[i]$.
We let $\alpha = u+iv$, where $u,v \in \ZZ$.

{\bf Case 1:} The unit $\epsilon = \pm 1$. 
We equate real and imaginary parts of equation~\eqref{eqn:desci} to obtain the equations:
\[
\begin{cases}
x+2r &= \pm (u^4 - 6u^2v^2 + v^4),  \\
x +r &= \pm 4(u^3v - uv^3).
\end{cases}
\]
Subtracting one from the other, we get: 
\[
\pm r = u^4 - 4u^3v - 6u^2v^2 +4uv^3 + v^4. 
\]

{\bf Case 2:} The unit $\epsilon = \pm i$. Equating real and imaginary parts of equation~\eqref{eqn:desci}, we obtain:
\[
\begin{cases}
x+2r &= \pm 4(-u^3v + uv^3), \\
x + r&= \pm (u^4 - 6u^2v^2 + v^4) .
\end{cases}
\]
Subtracting one from the other, we get:
\[
\pm r = u^4 + 4u^3v - 6u^2v^2 - 4uv^3 + v^4. 
\]

In both cases, when $r$ has a fixed value, we obtain homogeneous equations of degree 4. Using \texttt{Magma's} Thue solver, we determine all integer solutions $(u,v)$, whereby we recover the integer solutions $(x,|y|,4)$ to equation~\eqref{eq:main} for $d=2$.  
%for $n=4$ and a fixed value of $r$. 
These are recorded in Table~\ref{table:solutionsAP24}.

%------------------------------------------
\section{Primitive prime divisors of Lucas and Lehmer sequences} {\label{section: Primitive prime divisors of Lucas and Lehmer sequences}}
%------------------------------------------
%A \textit{Lucas pair} is a pair of algebraic integers $\alpha,\beta$,  such that $\alpha +\beta$ and $\alpha\beta$ are non--zero coprime rational integers and $\alpha/\beta$ is not a root of unity. The \textit{Lucas sequence} 
%associated to the Lucas pair $(\alpha,\beta)$ is
%\begin{equation*}
%u_n=u_n(\alpha,\beta)=
%\frac{\alpha^n-\beta^n}{\alpha-\beta}, \quad  \text{for }n \in \mathbb{N}.
%\end{equation*}
%A prime $p$ is called a \textit{primitive divisor} of $u_n$ if
%it divides $u_n$ but does not divide \\
%$(\alpha-\beta)^2\cdot u_1\cdots u_{n-1}$. 

A \textit{Lehmer pair} is a pair of algebraic integers $\alpha,\beta$,  such that $(\alpha +\beta)^2$ and $\alpha\beta$ are non--zero coprime rational integers and $\alpha/\beta$ is not a root of unity. The \textit{Lehmer sequence} 
associated to the Lehmer pair $(\alpha,\beta)$ is
\begin{equation*}
\tu_n=\tu_n(\alpha,\beta)=\begin{cases}
\frac{\alpha^n-\beta^n}{\alpha-\beta}, & \text{if $n$ is odd},\\
\frac{\alpha^n-\beta^n}{\alpha^2-\beta^2}, & \text{if $n$ is even}.
\end{cases}
\end{equation*}
A prime $p$ is called a \textit{primitive divisor} of $\tu_n$ if
it divides $\tu_n$ but does not divide 
$(\alpha^2-\beta^2)^2\cdot\tu_1\cdots\tu_{n-1}$. 
We shall make use of the following celebrated theorem~\cite{BHV}.
\begin{Th}[Bilu, Hanrot and Voutier]\label{thm:non_defective}
Let $\alpha$, $\beta$ be a Lehmer pair. Then
$\tu_n(\alpha,\beta)$ has a primitive divisor for all $n>30$,
and for all prime $n>13$.
\end{Th}

%------------------------------------------
%------------------------------------------
\section{An arithmetic progression with two terms} \label{sec:2AP}
%------------------------------------------
%------------------------------------------

%In this section, we find all integer solutions to equation~\eqref{eq:main} for $d=2$, $2\leq r \leq \bound$ and $n$ an odd prime. The case $d=2, r=1$ is solved in \cite{Cohn} and therefore we assume that $2\leq r \leq \bound$ for this section. We explain in detail in remark~\ref{rem:failure_2AP_r=1} why the method used below will not yield results for $r=1$.
In this section, we find all integer solutions to equation~\eqref{eq:main} for $d=2$, $1\leq r \leq \bound$ and $n$ an odd prime. 
We rewrite equation~\eqref{eq:main}  as  
\begin{equation*}%\label{eqn:AP2}
2x^2 + 6xr  + 5r^2 =y^n.
\end{equation*}

Multiplying by $2$ and completing the square, we obtain:
\begin{equation}\label{eqn:AP2}
(2x + 3r)^2 + r^2 = 2y^n.
\end{equation}

We apply the following general theorem.
\begin{Th}[Theorem 1 in \cite{AMLST}]\label{thm:AMLST}
Let $C$ be a positive integer satisfying $C \equiv 1 \pmod{4}$ and write 
$C = cd^2$ where $c$ is square--free. Suppose that $(x,y)$ is a solution to the equation
\begin{equation*}
x^2 + C = 2y^p, \quad x,~y \in \ZZ^{+}, \quad \gcd(x,y)=1,
\end{equation*}
where $p\geq 5$ is a prime. Then either,
\begin{enumerate}[(i)]
\item $x=y=C=1$, or
\item $p$ divides the class number of $\QQ(\sqrt{-c})$, or
\item $p=5$ and $(C,x,y) = (9,79,5), (125,19,3), (125,183,7), (2125,21417,47)$, or
\item $p \mid \left( q - \left(\frac{-c}{q}\right)\right)$, where $q$ is some odd prime such that $q\mid d$ and $q\nmid c$.
\end{enumerate}

\end{Th}

\subsection{Proof of Theorem~\ref{thm:main} for $d=2$}
%: $d=2$, $n$ is an odd prime}
%Since $\gcd(x,r)=1$ we can immediately
%deduce that $2\nmid r$.
We rewrite \eqref{eqn:AP2} as 
\[
\lvert 2x+3r\rvert^2+r^2=2y^n
\]
and apply Theorem~\ref{thm:AMLST}. Case (i)
gives the solutions $x=-1$ or $-2$, $r=1$, $y=1$ and $n$ arbitrary.
We suppose we are not in this case.
Let
\begin{equation*}
B = \{3,5\} \cup \left\{\text{$p$ odd prime} \; :\; p \mid \left( q - \left(\frac{-c}{q}\right)\right),\, \text{for some odd prime $q \mid r$}\right\}. %, 2 \nmid p,q\right\}.
\end{equation*} 
Note that if $r=1$ then $B=\{3,5\}$.
%Theorems~\ref{thm:non_defective} and 
Theorem~\ref{thm:AMLST} asserts  
%that for equation~\eqref{eqn:AP2}, 
%with $r \geq 2$, we have 
$n \in B$. Thus for every $2 \le r \le \bound$ we have finitely many possible
values of the prime exponent $n$.
We will explain how to solve \eqref{eqn:AP2} for a fixed $r$
and fixed exponent $n$. From \eqref{eqn:AP2} we obtain
\[
2x+3r+i r= (1+i) \alpha^n
\]
for some $\alpha \in \ZZ[i]$. Subtracting this equation from
its conjugate gives
\[
(1+i)\alpha^n-(1-i) \bar{\alpha}^n=2r i.
\]
Dividing by $1+i$ we have
\begin{equation}\label{eqn:thue}
\alpha^n+i \bar{\alpha}^n=(1+i)r.
\end{equation}
Let $\alpha=u+iv$ with $u$, $v \in \ZZ$. If $n \equiv 1 \pmod{4}$
then $i=i^n$ and if $n \equiv -1 \pmod{4}$ then $i=(-i)^n$. 
In the former case $\alpha+i\bar{\alpha}=(1+i)(u+v)$ is a factor
of the left-hand side of \eqref{eqn:thue}, and in the latter case
$\alpha-i\bar{\alpha}=(1-i)(u-v)$ is a factor. We deduce that $(u+v) \mid r$
or $(u-v) \mid r$ according to whether $n \equiv 1$ or $-1 \pmod{4}$.
Thus for each $1 \le r \le 10^4$ and for each $n \in B$ and for
each $t \mid r$, we let $u\pm v=t$, and we need to simply solve 
for $u$, $v$. But equation \eqref{eqn:thue} is now a polynomial
equation $v$ after letting $\alpha=u+iv=(t\mp v)+iv$.
We wrote a simple \texttt{Magma} script that solved these polynomial
equations and deduced the corresponding solutions to \eqref{eqn:AP2}.
This gives the solutions $(x,y,n)$ as in Table~\ref{table:solutionsAP2}. 
\begin{comment}
We wrote a simple \texttt{Magma} \cite{Magma} script which for
each $2 \le r \le \bound$ such that  $2 \nmid r$, and for each odd prime $v \mid r$, computed the set $B$.
For each odd prime $n \in B$
we know from equation~\eqref{eqn:2APsubconj} that $u$ is a root of 
%the polynomial
\begin{equation*}%\label{eqn:2APpoly}
\frac{1}{2 \cdot r \cdot \sqrt{-1} \cdot 2^{(n-1)/2}}
\cdot \left( (u+v \sqrt{-1})^n-(u-v\sqrt{-1})^n\right) \; -1.
\end{equation*}
\end{comment}

%------------------------------------------
%------------------------------------------
\section{An arithmetic progression with six terms} \label{sec:6AP}
%------------------------------------------
%------------------------------------------

In this section, we find all integer solutions to equation~\eqref{eq:main} for $d=6$ and $n$ an odd prime. 
%The case $n=2$ yielded no solutions in section~\ref{sec:n=2}; we can henceforth assume that $n$ is an odd prime. 
We rewrite equation~\eqref{eq:main} as
\begin{equation}{\label{eqn:case d=6}}
X^2 + 3\cdot 5\cdot 7 r^2 = 6y^n,
\end{equation}
%\begin{equation}{\label{eqn:case d=6}}
%(6x+15r)^2 + 3\cdot 5\cdot 7 r^2 = 6y^n.
%\end{equation}
where we let $X = 6x + 21r$ for ease of notation. We note here that $2,3 \nmid r$ else we contradict the assumption that $(x,y)$ is primitive.
%Since $(x,y)$ is primitive, we have $ 2, 3, 5, 7 \nmid y$. 
Let $K=\QQ(\sqrt{-105})$ and  its ring of integers, $\OK = \ZZ[\sqrt{-105}]$. This has class group isomorphic to $(\ZZ/2\ZZ)^3$. We can factorise equation~\eqref{eqn:case d=6} in $\OK$ as follows
\[
(X + r\sqrt{-105})(X-r\sqrt{-105})=6y^n.
\]
Let us write $\fp_2$ and $\fp_3$ for the prime ideals above 2 and 3, respectively. Let $\fa=\fp_2 \fp_3$. We write
\[
(X+ r\sqrt{-105})\OK  
%= \fp_2 \fp_3 \cdot \fz^n
 = \fa^{1-n} \cdot (\fa\fz)^n
 = (6^{(1-n)/2})  (\fa\fz)^n,
\]
%\begin{align*}
%(X+ r\sqrt{-105})\OK & = \fp_2 \fp_3 \cdot \fz^n\\
%& = \fa^{1-n} \cdot (\fa\fz)^n\\
%& = (6^{(1-n)/2}) \cdot (\fa\fz)^n\\
%\end{align*}
where $\fa \fz$ is a principal ideal of $\OK$. Indeed,  $[\fa\fz]^n = 1$ in the class group. Therefore the class $[\fa\fz]$ has order dividing $n$ or 1, as $n$ is an odd prime. Since the class group has order $8$, it means that the order of $[\fa \fz]$ must be 1. 
%hence it is principal. 
We therefore write $\fa \fz =(\gamma)\OK$ where $\gamma = u + v \sqrt{-105}\in \OK$ with $u,v \in \ZZ$. %We note that $\ord_{\fp_2}(\gamma)=\ord_{\fp_3}(\gamma)=1$. 
If required, we may swap $\gamma$ with $-\gamma$ to obtain
\begin{equation}{\label{eqn:individual factor}}
X + r \sqrt{-105} = \frac{\gamma^n}{6^{(n-1)/2}}.
\end{equation}
Subtracting the conjugate equation from the one above, we get
\begin{equation}{\label{eqn:subtract conjugate}}
\frac{\gamma^n}{6^{(n-1)/2}} - \frac{{\bar\gamma}^n}{6^{(n-1)/2}} = 2r\sqrt{-105},
\end{equation}
or equivalently,
\begin{equation*}
\frac{\gamma^n}{6^{n/2}} - \frac{{\bar\gamma}^n}{6^{n/2}} = r\sqrt{-70}.
\end{equation*}
Consider a quadratic extension, $L/K$, where $L=\QQ(\sqrt{-105}, \sqrt{6}) = \QQ(\sqrt{-70}, \sqrt{6})$. We write $\OL$ for its ring of integers and set
$
\alpha = \gamma/\sqrt{6} ,\;  \beta= \bar{\gamma}/\sqrt{6}.
$
Thus equation~\eqref{eqn:subtract conjugate} becomes
\begin{equation}{\label{eqn:6APalphabeta}}
\alpha^n - \beta^n = r\sqrt{-70}.
\end{equation}

\begin{Lemma}{\label{lem:6APLehmer}}
Let $\alpha,\beta$ be as above. Then $\alpha$ and $\beta$ are algebraic integers. Moreover, $(\alpha+\beta)^2$ and $\alpha\beta$ are non--zero, coprime, rational integers and $\alpha/\beta$ is not a unit.
\end{Lemma}

\begin{proof} We observe that $\fa \cdot \OL =\sqrt{6}\OL$. By definition, $\fa\mid \gamma, \bar{\gamma}$ and hence $\alpha, \beta$ are algebraic integers. Let $\gamma=u+v\sqrt{-105}$ with $u,v\in\ZZ$. Then
\[
(\alpha+\beta)^2=\frac{2u^2}{3}.
\]
%To show $(\alpha+\beta)^2\in \ZZ$, we must show that $3\mid u$. Indeed, since $\fp_3\mid \gamma, \sqrt{-105}$ we have $\fp_3\mid u$ and so $3\mid u$.
Since $\fp_3\mid \gamma, \sqrt{-105}$ we have $\fp_3\mid u$ and so $3\mid u$. Hence, $(\alpha+\beta)^2\in \ZZ$.
If $(\alpha+\beta)^2=0$ then $u=0$. However, from equation~\eqref{eqn:individual factor} and the fact that $n$ is odd, we obtain $X = 6x + 21 r=0$, hence $2x = -7r$. This contradicts the pairwise co-primality of $x,y,r$.
Thus $(\alpha + \beta)^2$ is a non-zero rational integer.
Moreover, $\alpha\beta=\gamma\bar{\gamma}/6$ is a non--zero rational integer since $3 \mid u$ and $\fp_2 \mid \gamma, \bar{\gamma}$.

We now check that $(\alpha+\beta)^2$ and $\alpha\beta$ are coprime. Suppose they are not coprime. Then there exists a prime $\fq$ of $\OL$ which divides both. 
Then $\fq$ divides $\alpha,\beta$ and from equations~\eqref{eqn:individual factor},\eqref{eqn:case d=6} and \eqref{eqn:6APalphabeta}, we see that $\fq^n$ divides $(y^n)\OL$ and $(r\sqrt{-70})\OL$. Since $\ord_{\fq}(r\sqrt{-70}) \geq n$, with $n$ an odd prime,  we contradict our assumption of $(x,y)$ being primitive.
%If $\fq$ divides $(6)\OL$, then since $\ord_{\fq}(r\sqrt{-70}) \geq n$, with $n$ an odd prime,  we contradict $2,3 \nmid r$. Else, we contradict our assumption of $(x,y)$ being primitive.

Finally, we show that $\alpha/\beta=\gamma/\bar\gamma\in\OK$ is not a unit. If it were so, then since the units in $K$ are $\pm 1$ we obtain $\alpha=\pm\beta$. This implies that either $u=0$ or $v=0$. We have seen earlier that we cannot have $u=0$. Substituting $v=0$ into equation~\eqref{eqn:individual factor}, we obtain $r=0$ and again arrive at a contradiction.
\end{proof}

Lemma~\ref{lem:6APLehmer} tells us that $(\alpha,\beta)$ is indeed a Lehmer pair. We denote by $\tu_k$ the associated Lehmer sequence. We may rewrite equation~\eqref{eqn:6APalphabeta} as
%So we know from above the following two equations
%\[
%\frac{\alpha -\beta}{\sqrt{-70}}=v\in \ZZ
%\]
%
%\[
%\frac{\alpha^p -\beta^p}{\sqrt{-70}} = r\in \ZZ
%\]
%combining which we get
\begin{equation*}
\left( \frac{\alpha^n -\beta^n}{\alpha-\beta} \right)\left( \frac{\alpha -\beta}{\sqrt{-70}}\right) = r.
%= r^\prime \cdot v \in \ZZ
\end{equation*}

Hence, we have 
\begin{equation}{\label{eqn:r prime}}
\frac{\alpha^n -\beta^n}{\alpha-\beta} = \frac{r}{v}=r^\prime.
\end{equation}

\begin{Lemma}\label{lem:B6}
Suppose $n > 13$. Then there is a prime $q \mid r$ such that $q \nmid 210$, 
and $n \mid B_q$ where 
\[
B_q=\begin{cases}
q-1 & \text{if $\left(\frac{-105}{q}\right)=1$}\\
q+1 & \text{if $\left(\frac{-105}{q}\right)=-1$}.
\end{cases}
\]
%Let
%\[
%B:=\max\left(
%\{13\} \cup \{
%B_q \; : \; \text{$q$ prime, $q\mid r^\prime$, $q \nmid 210v$}
%\}
%\right).
%\]
%Then $n \mid B_q$.
\end{Lemma}

\begin{proof}
Let $n>13$. By Theorem~\ref{thm:non_defective},
$\tilde{u}_n=(\alpha^n-\beta^n)/(\alpha-\beta)=r^\prime$
is divisible by a prime $q$ not dividing 
$(\alpha^2-\beta^2)^2=-280 u^2 v^2/3$ nor the terms
$\tilde{u}_1,\tilde{u}_2,\dotsc,\tilde{u}_{n-1}$.
We note that this is a prime $q$ dividing $r^\prime$
but not $210v$. Let $\fq$ be a prime of $K$
above $q$. As $(\alpha+\beta)^2$ and $\alpha\beta$ are coprime integers,
and as $\alpha$, $\beta$ satisfy equation~\eqref{eqn:r prime} we see that
$\fq \nmid \gamma$, $\overline{\gamma}$. We make two claims:
\begin{enumerate}
\item[(i)] the 
the multiplicative order
of the reduction of $\gamma/\overline{\gamma}$ modulo $\fq$
is $n$;
\item[(ii)] the multiplicative order
of the reduction of $\gamma/\overline{\gamma}$ modulo $\fq$
divides $B_q$.
\end{enumerate}
It follows immediately that $n \mid B_q$ which is what we
want to prove. 
We now need only prove (i), (ii). 
Let $m$ be a positive integer. 
Note that $\alpha/\overline{\alpha}=\gamma/\overline{\gamma}$.
Thus $q \mid \tilde{u}_m$
if and only if $(\gamma/\overline{\gamma})^m \equiv 1 \pmod{\fq}$.
Thus (i) follows as $q$ is a primitive divisor of $\tilde{u}_n$.
Let us prove (ii).
 If $-105$ is a square modulo $q$, then $\F_\fq=\F_q$
and so the multiplicative order divides $q-1=B_q$.
Otherwise, $\F_\fq=\F_{q^2}$. However, $\gamma/\overline{\gamma}$
has norm $1$, and the elements of norm $1$ in $\F_{q^2}^*$
form a subgroup of order $q+1=B_q$. In either case, the order
divides $B_q$. 
\begin{comment}
\margnote{Samir: You don't need BLMS. It just
complicates things!}
\[
(\gamma/\overline{\gamma})^{B_q} \equiv 1 \pmod{\fq}.
\]
Thus, $q \mid \tilde{u}_{B_q}$. However, 
%Moreover, $q \mid u_{B_q}$, where we define
%\[
%u_n = \frac{\gamma^n - \bar{\gamma}^n}{\gamma - \bar{\gamma}}.
%\]
%We notice that $u_n$ is indeed a Lucas sequence.

{\color{red}
Maybe we need this to come earlier and a proof to accompany?
}

Applying Lemma 5 from \cite{BLMS}, we conclude that $n \mid B_q$.
This
proves the lemma.
\end{comment}
\end{proof}
\begin{comment}
{
\color{red}

First we begin with some notation.
Let $r,s$ be non--zero integers with $\Delta = r^2 + 4s \neq 0$. Let $\alpha,\beta$ be roots of the equation $x^2 - rx - s = 0$ with the convention that $|\alpha|\geq |\beta|$.
We define the Lucas sequence $\{u_n\}_{n \geq 0}$ with parameters $r,s$ to be the sequence $u_n = (\alpha^n - \beta^n)/(\alpha - \beta)$.

If $T = \{l_1, \ldots, l_t\}$ is a finite set of primes, we write $\bar{T}$ for the set of integers of the form $\pm l_1^{x_1}l_2^{x_2}\cdots l_t^{x_t}$ with $x_i \ge 0$.

$S$ is the set of prime factors of $2\Delta$.

$S_1$ is the set of prime factors of $\gcd(r,s)$.

$S_2$ is the set of prime factors of $\Delta$ which are not in $S_1$.

$K = \QQ(\alpha) = \QQ(\beta)$ which is either $\QQ$ or a quadratic extension of $\QQ$.

$O_K$ are the ring of integers of $K$.

If $k\mid n$ are positive integers then $u_{n,k} = u_n / u_k$. Clearly, $u_{n,k}$ is a rational integer.

\begin{Lemma}[Lemma 5 in \cite{BLMS}]
Suppose $n$ is an integer and let $q$ be its smallest prime factor. Write $n=kq$. Then $\gcd(u_k,u_{n,k}) \in \bar{S}$.
\end{Lemma}

\begin{Lemma}[Lemma 2 in \cite{BLMS}]
Suppose that $n=kq$ where $n,k$ and $q$ are positive integers. Then $\gcd(u_k,u_{n,k})$ divides $q\epsilon$ for some $\epsilon \in \bar{S}$.
\end{Lemma}

For us, we know that $q\mid u_{B_q}$. In Lemma 2 let us take $k=B_q$. Now the $\gcd(u_{B_q}, u_{n,B_q})$ divides $q\epsilon$. Then it is kind of clear that $n \mid B_q$, for one, we need $u_{n,B_q}$ to be well defined.

}
\end{comment}
\subsection{Proof of Theorem~\ref{thm:main} for $d=6$}
%: $d=6$, $n$ is an odd prime}
%Since $\gcd(x,r)=1$ we deduce that $2,3\nmid r$.
%We wrote a \texttt{Magma} \cite{Magma} script which for
%each $1 \le r \le 5000$ such that  $2,3 \nmid r$, and for each $v \mid r$ %computes
%$B$ as in Lemma~\ref{lem:B6}. For each odd prime $n \le B$
Let 
\begin{equation*}
B = \{3,5,7,11,13\} \cup \left\{\text{$p$ odd prime} \; : \; 
p \mid \left( q - \left(\frac{-c}{q}\right)\right) \, \text{for some odd prime $q \mid r$}
\right\}.
%, 2 \nmid p,q\right\}.
\end{equation*} 
%Theorem~\ref{thm:non_defective} and 
Lemma~\ref{lem:B6} asserts  that for equation~\eqref{eqn:case d=6}, with $r \geq 1$, we have $n \in B$.
We wrote a simple \texttt{Magma} \cite{Magma} script which for
each $1 \le r \le \bound$ such that  $2,3 \nmid r$, and for each odd prime $v \mid r$, computed the set $B$.
For each odd prime $n \in B$
we know from \eqref{eqn:subtract conjugate} that $u$ is a root of \begin{equation*}%\label{eqn:6APpoly}
\frac{1}{2 \cdot r \cdot \sqrt{-105} \cdot 6^{(n-1)/2}}
\cdot \left( (u+v \sqrt{-105})^n-(u-v\sqrt{-105})^n\right) \; -1.
\end{equation*}
Computing these roots, we obtain the solutions $(x,y,n)$ as in Table~\ref{table:solutionsAP6}.

%------------------------------------------------------------------------
%------------------------------------------------------------------------
%------------------------------------------------------------------------
%------------------------------------------------------------------------
%\begin{appendix}
%------------------------------------------
%------------------------------------------
\section{Tables of Solutions}\label{sec:solntables}
%------------------------------------------
%------------------------------------------

\begin{center}
\begin{longtable}{|c|l|}
\caption{Triples of non-trivial primitive solutions $(x,y,n)$ 
\\
of equation~\eqref{eq:main} for $d=2$ and prime $n\geq3$ for  $1 \leq r\leq \bound$.}
  \label{table:solutionsAP2}
\\
\hline \multicolumn{1}{|c|}{\boldmath$r$} & \multicolumn{1}{c|}{\boldmath$(x,y,n)$}
\\ \hline 
\endfirsthead

\multicolumn{2}{c}
{{\bfseries \tablename\ \thetable{} -- continued from previous page}} \\
\hline \multicolumn{1}{|c|}{\boldmath$r$} &
\multicolumn{1}{c|}{\boldmath$(x,y,n)$}  \\ \hline 
\endhead

\hline \multicolumn{2}{|r|}{{Continued on next page}} \\ \hline
\endfoot

\hline \hline
\endlastfoot
%Magma V2.19-10    Thu Aug  9 2018 20:40:40 on Vanditas-Air [Seed = 3372157699]
%Type ? for help.  Type <Ctrl>-D to quit.
$ 1 $ &
%$ ( -1 , \pm1 , 2j ) $  ,
%$ ( -2 , \pm1 , 2j ) $  ,
$ ( -1 , 1 , n ) $  ,
$ ( -2 , 1 , n ) $  
%$ ( -1 , 1 , 2j+1 ) $  ,
%$ ( -2 , 1 , 2j+1 ) $  
\\ \hline
$ 3 $ &
$ ( -41 , 5 , 5 ) $  ,
$ ( 38 , 5 , 5 ) $
\\ \hline
$ 5 $ &
$ ( 47 , 17 , 3 ) $  ,
$ ( -52 , 17 , 3 ) $
\\ \hline
$ 9 $ &
$ ( 2 , 5 , 3 ) $  ,
$ ( -11 , 5 , 3 ) $
\\ \hline
$ 13 $ &
$ ( -11 , 5 , 3 ) $  ,
$ ( -2 , 5 , 3 ) $
\\ \hline
$ 19 $ &
$ ( 2636 , 241 , 3 ) $  ,
$ ( -2655 , 241 , 3 ) $
\\ \hline
$ 27 $ &
$ ( 259 , 53 , 3 ) $  ,
$ ( -286 , 53 , 3 ) $
\\ \hline
$ 37 $ &
$ ( -46 , 13 , 3 ) $  ,
$ ( 9 , 13 , 3 ) $
\\ \hline
$ 55 $ &
$ ( -9 , 13 , 3 ) $  ,
$ ( -46 , 13 , 3 ) $
\\ \hline
$ 71 $ &
$ ( 137745 , 3361 , 3 ) $  ,
$ ( -137816 , 3361 , 3 ) $
\\ \hline
$ 73 $ &
$ ( -117 , 25 , 3 ) $  ,
$ ( 44 , 25 , 3 ) $
\\ \hline
$ 77 $ &
$ ( 65 , 29 , 3 ) $  ,
$ ( -142 , 29 , 3 ) $
\\ \hline
$ 79 $ &
$ ( -38 , 5 , 5 ) $  ,
$ ( -41 , 5 , 5 ) $
\\ \hline
$ 91 $ &
$ ( 107 , 37 , 3 ) $  ,
$ ( -198 , 37 , 3 ) $
\\ \hline
$ 99 $ &
$ ( -47 , 17 , 3 ) $  ,
$ ( -52 , 17 , 3 ) $  ,
$ ( 13754 , 725 , 3 ) $  ,
$ ( -13853 , 725 , 3 ) $
\\ \hline
$ 121 $ &
$ ( -236 , 41 , 3 ) $  ,
$ ( 115 , 41 , 3 ) $
\\ \hline
$ 143 $ &
$ ( 478 , 85 , 3 ) $  ,
$ ( -621 , 85 , 3 ) $  ,
$ ( 730 , 109 , 3 ) $  ,
$ ( -873 , 109 , 3 ) $
\\ \hline
$ 161 $ &
$ ( -44 , 25 , 3 ) $  ,
$ ( -117 , 25 , 3 ) $
\\ \hline
$ 181 $ &
$ ( -415 , 61 , 3 ) $  ,
$ ( 234 , 61 , 3 ) $
\\ \hline
$ 207 $ &
$ ( -142 , 29 , 3 ) $  ,
$ ( -65 , 29 , 3 ) $
\\ \hline
$ 249 $ &
$ ( 29 , 5 , 7 ) $  ,
$ ( -278 , 5 , 7 ) $
\\ \hline
$ 253 $ &
$ ( -666 , 85 , 3 ) $  ,
$ ( 413 , 85 , 3 ) $  ,
$ ( 296 , 73 , 3 ) $  
\\ \hline
 & 
$ ( -549 , 73 , 3 ) $  ,
$ ( 4482 , 349 , 3 ) $  ,
$ ( -4735 , 349 , 3 ) $
\\ \hline
$ 265 $ &
$ ( 7162792 , 46817 , 3 ) $  ,
$ ( -7163057 , 46817 , 3 ) $
\\ \hline
$ 297 $ &
$ ( 191 , 65 , 3 ) $  ,
$ ( -488 , 65 , 3 ) $
\\ \hline
$ 305 $ &
$ ( -107 , 37 , 3 ) $  ,
$ ( -198 , 37 , 3 ) $
\\ \hline
$ 307 $ &
$ ( -29 , 5 , 7 ) $  ,
$ ( -278 , 5 , 7 ) $
\\ \hline
$ 337 $ &
$ ( -1001 , 113 , 3 ) $  ,
$ ( 664 , 113 , 3 ) $
\\ \hline
$ 351 $ &
$ ( -115 , 41 , 3 ) $  ,
$ ( -236 , 41 , 3 ) $
\\ \hline
$ 369 $ &
$ ( 715957 , 10085 , 3 ) $  ,
$ ( -716326 , 10085 , 3 ) $
\\ \hline
$ 377 $ &
$ ( 9306 , 565 , 3 ) $  ,
$ ( -9683 , 565 , 3 ) $
\\ \hline
$ 391 $ &
$ ( 1573 , 185 , 3 ) $  ,
$ ( -1964 , 185 , 3 ) $
\\ \hline
$ 433 $ &
$ ( -1432 , 145 , 3 ) $  ,
$ ( 999 , 145 , 3 ) $
\\ \hline
$ 475 $ &
$ ( 5646 , 37 , 5 ) $  ,
$ ( -597 , 13 , 5 ) $  ,
$ ( 122 , 13 , 5 ) $  ,
$ ( -6121 , 37 , 5 ) $
\\ \hline
$ 481 $ &
$ ( 718 , 125 , 3 ) $  ,
$ ( -1199 , 125 , 3 ) $
\\ \hline
$ 517 $ &
$ ( 7 , 65 , 3 ) $  ,
$ ( -524 , 65 , 3 ) $  ,
$ ( 39553 , 1469 , 3 ) $  ,
$ ( -40070 , 1469 , 3 ) $
\\ \hline
$ 531 $ &
$ ( -524 , 65 , 3 ) $  ,
$ ( -7 , 65 , 3 ) $
\\ \hline
$ 541 $ &
$ ( -1971 , 181 , 3 ) $  ,
$ ( 1430 , 181 , 3 ) $
\\ \hline
$ 545 $ &
$ ( -286 , 53 , 3 ) $  ,
$ ( -259 , 53 , 3 ) $
\\ \hline
$ 559 $ &
$ ( 30483 , 1237 , 3 ) $  ,
$ ( -31042 , 1237 , 3 ) $
\\ \hline
$ 585 $ &
$ ( 803 , 137 , 3 ) $  ,
$ ( -1388 , 137 , 3 ) $
\\ \hline
$ 611 $ &
$ ( 297 , 97 , 3 ) $  ,
$ ( -908 , 97 , 3 ) $
\\ \hline
$ 629 $ &
$ ( 1737 , 205 , 3 ) $  ,
$ ( -2366 , 205 , 3 ) $
\\ \hline
$ 649 $ &
$ ( -234 , 61 , 3 ) $  ,
$ ( -415 , 61 , 3 ) $
\\ \hline
$ 661 $ &
$ ( -2630 , 221 , 3 ) $  ,
$ ( 1969 , 221 , 3 ) $
\\ \hline
$ 671 $ &
$ ( 299 , 101 , 3 ) $  ,
$ ( -970 , 101 , 3 ) $
\\ \hline
$ 679 $ &
$ ( -191 , 65 , 3 ) $  ,
$ ( -488 , 65 , 3 ) $
\\ \hline
$ 693 $ &
$ ( 3404 , 305 , 3 ) $  ,
$ ( -4097 , 305 , 3 ) $
\\ \hline
$ 717 $ &
$ ( 404 , 17 , 5 ) $  ,
$ ( -1121 , 17 , 5 ) $
\\ \hline
$ 719 $ &
$ ( -597 , 13 , 5 ) $  ,
$ ( -122 , 13 , 5 ) $
\\ \hline
$ 747 $ &
$ ( -835 , 89 , 3 ) $  ,
$ ( 88 , 89 , 3 ) $
\\ \hline
$ 793 $ &
$ ( -3421 , 265 , 3 ) $  ,
$ ( 2628 , 265 , 3 ) $
\\ \hline
$ 819 $ &
$ ( 2896 , 281 , 3 ) $  ,
$ ( -3715 , 281 , 3 ) $
\\ \hline
$ 845 $ &
$ ( -549 , 73 , 3 ) $  ,
$ ( -296 , 73 , 3 ) $
\\ \hline
$ 851 $ &
$ ( 18821 , 905 , 3 ) $  ,
$ ( -19672 , 905 , 3 ) $
\\ \hline
$ 923 $ &
$ ( -88 , 89 , 3 ) $  ,
$ ( -835 , 89 , 3 ) $
\\ \hline
$ 935 $ &
$ ( 235639 , 4813 , 3 ) $  ,
$ ( -236574 , 4813 , 3 ) $
\\ \hline
$ 937 $ &
$ ( -4356 , 313 , 3 ) $  ,
$ ( 3419 , 313 , 3 ) $
\\ \hline
$ 989 $ &
$ ( 372337471 , 652081 , 3 ) $  ,
$ ( -372338460 , 652081 , 3 ) $
\\ \hline
$ 1035 $ &
$ ( 1006 , 173 , 3 ) $  ,
$ ( -2041 , 173 , 3 ) $
\\ \hline
$ 1079 $ &
$ ( -413 , 85 , 3 ) $  ,
$ ( -666 , 85 , 3 ) $
\\ \hline
$ 1093 $ &
$ ( -5447 , 365 , 3 ) $  ,
$ ( 4354 , 365 , 3 ) $
\\ \hline
$ 1099 $ &
$ ( -478 , 85 , 3 ) $  ,
$ ( -621 , 85 , 3 ) $
\\ \hline
$ 1121 $ &
$ ( 1694 , 221 , 3 ) $  ,
$ ( -2815 , 221 , 3 ) $
\\ \hline
$ 1205 $ &
$ ( -908 , 97 , 3 ) $  ,
$ ( -297 , 97 , 3 ) $
\\ \hline
$ 1207 $ &
$ ( 828 , 169 , 3 ) $  ,
$ ( -2035 , 169 , 3 ) $
\\ \hline
$ 1261 $ &
$ ( -6706 , 421 , 3 ) $  ,
$ ( 5445 , 421 , 3 ) $  ,
$ ( 431 , 145 , 3 ) $  ,
$ ( -1692 , 145 , 3 ) $
\\ \hline
$ 1269 $ &
$ ( -299 , 101 , 3 ) $  ,
$ ( -970 , 101 , 3 ) $
\\ \hline
$ 1287 $ &
$ ( -1757 , 149 , 3 ) $  ,
$ ( 470 , 149 , 3 ) $  ,
$ ( 22552 , 1025 , 3 ) $  ,
$ ( -23839 , 1025 , 3 ) $
\\ \hline
$ 1377 $ &
$ ( 37219754 , 140453 , 3 ) $  ,
$ ( -37221131 , 140453 , 3 ) $
\\ \hline
$ 1387 $ &
$ ( 2277 , 265 , 3 ) $  ,
$ ( -3664 , 265 , 3 ) $
\\ \hline
$ 1403 $ &
$ ( 86256 , 2473 , 3 ) $  ,
$ ( -87659 , 2473 , 3 ) $
\\ \hline
$ 1417 $ &
$ ( 498157 , 7925 , 3 ) $  ,
$ ( -499574 , 7925 , 3 ) $
\\ \hline
$ 1441 $ &
$ ( -8145 , 481 , 3 ) $  ,
$ ( 6704 , 481 , 3 ) $
\\ \hline
$ 1457 $ &
$ ( 10296 , 625 , 3 ) $  ,
$ ( -11753 , 625 , 3 ) $
\\ \hline
$ 1475 $ &
$ ( 4807 , 397 , 3 ) $  ,
$ ( -6282 , 397 , 3 ) $
\\ \hline
$ 1525 $ &
$ ( -1121 , 17 , 5 ) $  ,
$ ( -404 , 17 , 5 ) $
\\ \hline
$ 1603 $ &
$ ( -873 , 109 , 3 ) $  ,
$ ( -730 , 109 , 3 ) $
\\ \hline
$ 1611 $ &
$ ( -2392 , 185 , 3 ) $  ,
$ ( 781 , 185 , 3 ) $
\\ \hline
$ 1633 $ &
$ ( -9776 , 545 , 3 ) $  ,
$ ( 8143 , 545 , 3 ) $
\\ \hline
$ 1665 $ &
$ ( -664 , 113 , 3 ) $  ,
$ ( -1001 , 113 , 3 ) $
\\ \hline
$ 1679 $ &
$ ( 4268 , 377 , 3 ) $  ,
$ ( -5947 , 377 , 3 ) $
\\ \hline
$ 1819 $ &
$ ( 143 , 157 , 3 ) $  ,
$ ( -1962 , 157 , 3 ) $
\\ \hline
$ 1837 $ &
$ ( -11611 , 613 , 3 ) $  ,
$ ( 9774 , 613 , 3 ) $
\\ \hline
$ 1853 $ &
$ ( 1342 , 229 , 3 ) $  ,
$ ( -3195 , 229 , 3 ) $
\\ \hline
$ 1863 $ &
$ ( 7135 , 509 , 3 ) $  ,
$ ( -8998 , 509 , 3 ) $
\\ \hline
$ 1891 $ &
$ ( 11034 , 661 , 3 ) $  ,
$ ( -12925 , 661 , 3 ) $
\\ \hline
$ 1909 $ &
$ ( 3077 , 325 , 3 ) $  ,
$ ( -4986 , 325 , 3 ) $
\\ \hline
$ 1917 $ &
$ ( -718 , 125 , 3 ) $  ,
$ ( -1199 , 125 , 3 ) $
\\ \hline
$ 1925 $ &
$ ( 2061306 , 20413 , 3 ) $  ,
$ ( -2063231 , 20413 , 3 ) $
\\ \hline
$ 1927 $ &
$ ( 24992 , 1105 , 3 ) $  ,
$ ( -26919 , 1105 , 3 ) $
\\ \hline
$ 1961 $ &
$ ( 56881 , 1885 , 3 ) $  ,
$ ( -58842 , 1885 , 3 ) $
\\ \hline
$ 1989 $ &
$ ( 49480 , 1721 , 3 ) $  ,
$ ( -51469 , 1721 , 3 ) $
\\ \hline
$ 2033 $ &
$ ( 793 , 205 , 3 ) $  ,
$ ( -2826 , 205 , 3 ) $
\\ \hline
$ 2053 $ &
$ ( -13662 , 685 , 3 ) $  ,
$ ( 11609 , 685 , 3 ) $
\\ \hline
$ 2093 $ &
$ ( 1605854 , 17285 , 3 ) $  ,
$ ( -1607947 , 17285 , 3 ) $
\\ \hline
$ 2105 $ &
$ ( 4449 , 13 , 7 ) $  ,
$ ( -6554 , 13 , 7 ) $  ,
$ ( -1962 , 157 , 3 ) $  ,
$ ( -143 , 157 , 3 ) $
\\ \hline
$ 2115 $ &
$ ( 587 , 197 , 3 ) $  ,
$ ( -2702 , 197 , 3 ) $
\\ \hline
$ 2123 $ &
$ ( -431 , 145 , 3 ) $  ,
$ ( -1692 , 145 , 3 ) $
\\ \hline
$ 2191 $ &
$ ( -1388 , 137 , 3 ) $  ,
$ ( -803 , 137 , 3 ) $
\\ \hline
$ 2227 $ &
$ ( -470 , 149 , 3 ) $  ,
$ ( -1757 , 149 , 3 ) $
\\ \hline
$ 2281 $ &
$ ( -15941 , 761 , 3 ) $  ,
$ ( 13660 , 761 , 3 ) $
\\ \hline
$ 2367 $ &
$ ( -4070 , 269 , 3 ) $  ,
$ ( 1703 , 269 , 3 ) $
\\ \hline
$ 2407 $ &
$ ( 7414 , 533 , 3 ) $  ,
$ ( -9821 , 533 , 3 ) $
\\ \hline
$ 2431 $ &
$ ( -999 , 145 , 3 ) $  ,
$ ( -1432 , 145 , 3 ) $
\\ \hline
$ 2479 $ &
$ ( 115234 , 3005 , 3 ) $  ,
$ ( -117713 , 3005 , 3 ) $
\\ \hline
$ 2485 $ &
$ ( 14499 , 793 , 3 ) $  ,
$ ( -16984 , 793 , 3 ) $
\\ \hline
$ 2521 $ &
$ ( -18460 , 841 , 3 ) $  ,
$ ( 15939 , 841 , 3 ) $
\\ \hline
$ 2645 $ &
$ ( -2681 , 193 , 3 ) $  ,
$ ( 36 , 193 , 3 ) $
\\ \hline
$ 2673 $ &
$ ( 203983 , 4385 , 3 ) $  ,
$ ( -206656 , 4385 , 3 ) $
\\ \hline
$ 2717 $ &
$ ( -36 , 193 , 3 ) $  ,
$ ( -2681 , 193 , 3 ) $
\\ \hline
$ 2773 $ &
$ ( -21231 , 925 , 3 ) $  ,
$ ( 18458 , 925 , 3 ) $
\\ \hline
$ 2799 $ &
$ ( -5137 , 317 , 3 ) $  ,
$ ( 2338 , 317 , 3 ) $
\\ \hline
$ 2807 $ &
$ ( -4282 , 29 , 5 ) $  ,
$ ( 1475 , 29 , 5 ) $
\\ \hline
$ 2863 $ &
$ ( -2035 , 169 , 3 ) $  ,
$ ( -828 , 169 , 3 ) $
\\ \hline
$ 2879 $ &
$ ( -3116 , 25 , 5 ) $  ,
$ ( 237 , 25 , 5 ) $
\\ \hline
$ 2925 $ &
$ ( 160067 , 3737 , 3 ) $  ,
$ ( -162992 , 3737 , 3 ) $
\\ \hline
$ 2983 $ &
$ ( 1726 , 293 , 3 ) $  ,
$ ( -4709 , 293 , 3 ) $
\\ \hline
$ 2989 $ &
$ ( 44630 , 1621 , 3 ) $  ,
$ ( -47619 , 1621 , 3 ) $
\\ \hline
$ 2997 $ &
$ ( 5060 , 449 , 3 ) $  ,
$ ( -8057 , 449 , 3 ) $
\\ \hline
$ 3025 $ &
$ ( 5228 , 457 , 3 ) $  ,
$ ( -8253 , 457 , 3 ) $
\\ \hline
$ 3037 $ &
$ ( -24266 , 1013 , 3 ) $  ,
$ ( 21229 , 1013 , 3 ) $
\\ \hline
$ 3047 $ &
$ ( -1006 , 173 , 3 ) $  ,
$ ( -2041 , 173 , 3 ) $
\\ \hline
$ 3151 $ &
$ ( 987228 , 12505 , 3 ) $  ,
$ ( -990379 , 12505 , 3 ) $
\\ \hline
$ 3173 $ &
$ ( -781 , 185 , 3 ) $  ,
$ ( -2392 , 185 , 3 ) $
\\ \hline
$ 3245 $ &
$ ( -3544 , 233 , 3 ) $  ,
$ ( 299 , 233 , 3 ) $
\\ \hline
$ 3275 $ &
$ ( 3186 , 373 , 3 ) $  ,
$ ( -6461 , 373 , 3 ) $
\\ \hline
$ 3281 $ &
$ ( 767 , 257 , 3 ) $  ,
$ ( -4048 , 257 , 3 ) $
\\ \hline
$ 3289 $ &
$ ( -587 , 197 , 3 ) $  ,
$ ( -2702 , 197 , 3 ) $  ,
$ ( 1744 , 305 , 3 ) $  ,
$ ( -5033 , 305 , 3 ) $
\\ \hline
$ 3313 $ &
$ ( -27577 , 1105 , 3 ) $  ,
$ ( 24264 , 1105 , 3 ) $
\\ \hline
$ 3353 $ &
$ ( -3116 , 25 , 5 ) $  ,
$ ( -237 , 25 , 5 ) $
\\ \hline
$ 3401 $ &
$ ( -1430 , 181 , 3 ) $  ,
$ ( -1971 , 181 , 3 ) $
\\ \hline
$ 3487 $ &
$ ( 12258250 , 66989 , 3 ) $  ,
$ ( -12261737 , 66989 , 3 ) $
\\ \hline
$ 3509 $ &
$ ( 8515 , 601 , 3 ) $  ,
$ ( -12024 , 601 , 3 ) $
\\ \hline
$ 3537 $ &
$ ( -1964 , 185 , 3 ) $  ,
$ ( -1573 , 185 , 3 ) $
\\ \hline
$ 3601 $ &
$ ( -31176 , 1201 , 3 ) $  ,
$ ( 27575 , 1201 , 3 ) $
\\ \hline
$ 3619 $ &
$ ( -2826 , 205 , 3 ) $  ,
$ ( -793 , 205 , 3 ) $
\\ \hline
$ 3663 $ &
$ ( 10825 , 689 , 3 ) $  ,
$ ( -14488 , 689 , 3 ) $
\\ \hline
$ 3691 $ &
$ ( 19354423140 , 9082321 , 3 ) $  ,
$ ( -19354426831 , 9082321 , 3 ) $
\\ \hline
$ 3771 $ &
$ ( -7787 , 425 , 3 ) $  ,
$ ( 4016 , 425 , 3 ) $
\\ \hline
$ 3827 $ &
$ ( 2642 , 5 , 11 ) $  ,
$ ( -6469 , 5 , 11 ) $
\\ \hline
$ 3835 $ &
$ ( 102663 , 2797 , 3 ) $  ,
$ ( -106498 , 2797 , 3 ) $
\\ \hline
$ 3843 $ &
$ ( -299 , 233 , 3 ) $  ,
$ ( -3544 , 233 , 3 ) $
\\ \hline
$ 3887 $ &
$ ( 3573 , 409 , 3 ) $  ,
$ ( -7460 , 409 , 3 ) $
\\ \hline
$ 3901 $ &
$ ( -35075 , 1301 , 3 ) $  ,
$ ( 31174 , 1301 , 3 ) $
\\ \hline
$ 3905 $ &
$ ( -4563 , 277 , 3 ) $  ,
$ ( 658 , 277 , 3 ) $  ,
$ ( 60931 , 1993 , 3 ) $  ,
$ ( -64836 , 1993 , 3 ) $
\\ \hline
$ 3977 $ &
$ ( 25624 , 1153 , 3 ) $  ,
$ ( -29601 , 1153 , 3 ) $
\\ \hline
$ 4033 $ &
$ ( 18557 , 949 , 3 ) $  ,
$ ( -22590 , 949 , 3 ) $
\\ \hline
$ 4103 $ &
$ ( -1737 , 205 , 3 ) $  ,
$ ( -2366 , 205 , 3 ) $
\\ \hline
$ 4213 $ &
$ ( -39286 , 1405 , 3 ) $  ,
$ ( 35073 , 1405 , 3 ) $
\\ \hline
$ 4311 $ &
$ ( -9394 , 485 , 3 ) $  ,
$ ( 5083 , 485 , 3 ) $
\\ \hline
$ 4347 $ &
$ ( 95303 , 2669 , 3 ) $  ,
$ ( -99650 , 2669 , 3 ) $
\\ \hline
$ 4393 $ &
$ ( 495 , 289 , 3 ) $  ,
$ ( -4888 , 289 , 3 ) $
\\ \hline
$ 4433 $ &
$ ( 7751 , 593 , 3 ) $  ,
$ ( -12184 , 593 , 3 ) $
\\ \hline
$ 4473 $ &
$ ( 2158 , 365 , 3 ) $  ,
$ ( -6631 , 365 , 3 ) $
\\ \hline
$ 4509 $ &
$ ( -2815 , 221 , 3 ) $  ,
$ ( -1694 , 221 , 3 ) $
\\ \hline
$ 4537 $ &
$ ( -43821 , 1513 , 3 ) $  ,
$ ( 39284 , 1513 , 3 ) $  ,
$ ( -1342 , 229 , 3 ) $  ,
$ ( -3195 , 229 , 3 ) $
\\ \hline
$ 4599 $ &
$ ( -1969 , 221 , 3 ) $  ,
$ ( -2630 , 221 , 3 ) $
\\ \hline
$ 4775 $ &
$ ( 1649 , 353 , 3 ) $  ,
$ ( -6424 , 353 , 3 ) $
\\ \hline
$ 4779 $ &
$ ( 6284 , 545 , 3 ) $  ,
$ ( -11063 , 545 , 3 ) $
\\ \hline
$ 4807 $ &
$ ( 971 , 325 , 3 ) $  ,
$ ( -5778 , 325 , 3 ) $
\\ \hline
$ 4815 $ &
$ ( -767 , 257 , 3 ) $  ,
$ ( -4048 , 257 , 3 ) $
\\ \hline
$ 4843 $ &
$ ( 5229 , 505 , 3 ) $  ,
$ ( -10072 , 505 , 3 ) $
\\ \hline
$ 4851 $ &
$ ( 1221617 , 14417 , 3 ) $  ,
$ ( -1226468 , 14417 , 3 ) $
\\ \hline
$ 4873 $ &
$ ( -48692 , 1625 , 3 ) $  ,
$ ( 43819 , 1625 , 3 ) $
\\ \hline
$ 4941 $ &
$ ( 3211 , 425 , 3 ) $  ,
$ ( -8152 , 425 , 3 ) $
\\ \hline
$ 5139 $ &
$ ( 1934725219 , 1956245 , 3 ) $  ,
$ ( -1934730358 , 1956245 , 3 ) $
\\ \hline
$ 5221 $ &
$ ( -53911 , 1741 , 3 ) $  ,
$ ( 48690 , 1741 , 3 ) $  ,
$ ( -658 , 277 , 3 ) $  
\\ \hline
 & 
$ ( -4563 , 277 , 3 ) $  ,
$ ( 4498091 , 34345 , 3 ) $  ,
$ ( -4503312 , 34345 , 3 ) $
\\ \hline
$ 5243 $ &
$ ( 41078 , 1565 , 3 ) $  ,
$ ( -46321 , 1565 , 3 ) $
\\ \hline
$ 5251 $ &
$ ( 267190 , 5261 , 3 ) $  ,
$ ( -272441 , 5261 , 3 ) $
\\ \hline
$ 5291 $ &
$ ( -2636 , 241 , 3 ) $  ,
$ ( -2655 , 241 , 3 ) $  ,
$ ( 15148 , 865 , 3 ) $  
\\ \hline
 & 
$ ( -20439 , 865 , 3 ) $  ,
$ ( 25948522 , 110437 , 3 ) $  ,
$ ( -25953813 , 110437 , 3 ) $
\\ \hline
$ 5311 $ &
$ ( 551177 , 8497 , 3 ) $  ,
$ ( -556488 , 8497 , 3 ) $
\\ \hline
$ 5383 $ &
$ ( -4888 , 289 , 3 ) $  ,
$ ( -495 , 289 , 3 ) $
\\ \hline
$ 5405 $ &
$ ( -7117 , 377 , 3 ) $  ,
$ ( 1712 , 377 , 3 ) $
\\ \hline
$ 5499 $ &
$ ( -13232 , 617 , 3 ) $  ,
$ ( 7733 , 617 , 3 ) $
\\ \hline
$ 5581 $ &
$ ( -59490 , 1861 , 3 ) $  ,
$ ( 53909 , 1861 , 3 ) $
\\ \hline
$ 5611 $ &
$ ( 10015 , 701 , 3 ) $  ,
$ ( -15626 , 701 , 3 ) $
\\ \hline
$ 5621 $ &
$ ( 174512 , 3977 , 3 ) $  ,
$ ( -180133 , 3977 , 3 ) $
\\ \hline
$ 5633 $ &
$ ( 26037 , 1189 , 3 ) $  ,
$ ( -31670 , 1189 , 3 ) $
\\ \hline
$ 5723 $ &
$ ( 77796 , 2353 , 3 ) $  ,
$ ( -83519 , 2353 , 3 ) $
\\ \hline
$ 5725 $ &
$ ( 4329 , 493 , 3 ) $  ,
$ ( -10054 , 493 , 3 ) $
\\ \hline
$ 5757 $ &
$ ( -1475 , 29 , 5 ) $  ,
$ ( -4282 , 29 , 5 ) $
\\ \hline
$ 5773 $ &
$ ( -1703 , 269 , 3 ) $  ,
$ ( -4070 , 269 , 3 ) $
\\ \hline
$ 5941 $ &
$ ( -2277 , 265 , 3 ) $  ,
$ ( -3664 , 265 , 3 ) $
\\ \hline
$ 5953 $ &
$ ( -65441 , 1985 , 3 ) $  ,
$ ( 59488 , 1985 , 3 ) $
\\ \hline
$ 5975 $ &
$ ( 208 , 337 , 3 ) $  ,
$ ( -6183 , 337 , 3 ) $
\\ \hline
$ 6049 $ &
$ ( -2628 , 265 , 3 ) $  ,
$ ( -3421 , 265 , 3 ) $
\\ \hline
$ 6147 $ &
$ ( -15487 , 689 , 3 ) $  ,
$ ( 9340 , 689 , 3 ) $
\\ \hline
$ 6245 $ &
$ ( -8676 , 433 , 3 ) $  ,
$ ( 2431 , 433 , 3 ) $
\\ \hline
$ 6265 $ &
$ ( 11258 , 757 , 3 ) $  ,
$ ( -17523 , 757 , 3 ) $
\\ \hline
$ 6313 $ &
$ ( 74014 , 2285 , 3 ) $  ,
$ ( -80327 , 2285 , 3 ) $
\\ \hline
$ 6335 $ &
$ ( 14922 , 877 , 3 ) $  ,
$ ( -21257 , 877 , 3 ) $
\\ \hline
$ 6337 $ &
$ ( -71776 , 2113 , 3 ) $  ,
$ ( 65439 , 2113 , 3 ) $
\\ \hline
$ 6371 $ &
$ ( 2638 , 445 , 3 ) $  ,
$ ( -9009 , 445 , 3 ) $
\\ \hline
$ 6391 $ &
$ ( -6183 , 337 , 3 ) $  ,
$ ( -208 , 337 , 3 ) $
\\ \hline
$ 6435 $ &
$ ( -1726 , 293 , 3 ) $  ,
$ ( -4709 , 293 , 3 ) $
\\ \hline
$ 6557 $ &
$ ( 184574 , 4133 , 3 ) $  ,
$ ( -191131 , 4133 , 3 ) $
\\ \hline
$ 6611 $ &
$ ( -3715 , 281 , 3 ) $  ,
$ ( -2896 , 281 , 3 ) $
\\ \hline
$ 6643 $ &
$ ( 288629 , 5545 , 3 ) $  ,
$ ( -295272 , 5545 , 3 ) $
\\ \hline
$ 6733 $ &
$ ( -78507 , 2245 , 3 ) $  ,
$ ( 71774 , 2245 , 3 ) $
\\ \hline
$ 6741 $ &
$ ( 1199 , 401 , 3 ) $  ,
$ ( -7940 , 401 , 3 ) $
\\ \hline
$ 6749 $ &
$ ( -971 , 325 , 3 ) $  ,
$ ( -5778 , 325 , 3 ) $
\\ \hline
$ 6777 $ &
$ ( -5033 , 305 , 3 ) $  ,
$ ( -1744 , 305 , 3 ) $
\\ \hline
$ 6903 $ &
$ ( 21161 , 1073 , 3 ) $  ,
$ ( -28064 , 1073 , 3 ) $
\\ \hline
$ 6931 $ &
$ ( 3140 , 481 , 3 ) $  ,
$ ( -10071 , 481 , 3 ) $
\\ \hline
$ 6989 $ &
$ ( 7436 , 641 , 3 ) $  ,
$ ( -14425 , 641 , 3 ) $
\\ \hline
$ 7037 $ &
$ ( 5302 , 565 , 3 ) $  ,
$ ( -12339 , 565 , 3 ) $
\\ \hline
$ 7097 $ &
$ ( 1319591 , 15185 , 3 ) $  ,
$ ( -1326688 , 15185 , 3 ) $
\\ \hline
$ 7141 $ &
$ ( -85646 , 2381 , 3 ) $  ,
$ ( 78505 , 2381 , 3 ) $
\\ \hline
$ 7145 $ &
$ ( -10439 , 493 , 3 ) $  ,
$ ( 3294 , 493 , 3 ) $
\\ \hline
$ 7183 $ &
$ ( 107167895 , 284269 , 3 ) $  ,
$ ( -107175078 , 284269 , 3 ) $
\\ \hline
$ 7191 $ &
$ ( 49430 , 1781 , 3 ) $  ,
$ ( -56621 , 1781 , 3 ) $
\\ \hline
$ 7245 $ &
$ ( 444026 , 7373 , 3 ) $  ,
$ ( -451271 , 7373 , 3 ) $
\\ \hline
$ 7259 $ &
$ ( 647075 , 9461 , 3 ) $  ,
$ ( -654334 , 9461 , 3 ) $
\\ \hline
$ 7267 $ &
$ ( 28888 , 1289 , 3 ) $  ,
$ ( -36155 , 1289 , 3 ) $
\\ \hline
$ 7339 $ &
$ ( 20491 , 1061 , 3 ) $  ,
$ ( -27830 , 1061 , 3 ) $
\\ \hline
$ 7363 $ &
$ ( 3031686 , 26413 , 3 ) $  ,
$ ( -3039049 , 26413 , 3 ) $
\\ \hline
$ 7371 $ &
$ ( 2592749 , 23801 , 3 ) $  ,
$ ( -2600120 , 23801 , 3 ) $
\\ \hline
$ 7379 $ &
$ ( 38259 , 1525 , 3 ) $  ,
$ ( -45638 , 1525 , 3 ) $
\\ \hline
$ 7475 $ &
$ ( -2338 , 317 , 3 ) $  ,
$ ( -5137 , 317 , 3 ) $
\\ \hline
$ 7483 $ &
$ ( -7670 , 389 , 3 ) $  ,
$ ( 187 , 389 , 3 ) $
\\ \hline
$ 7501 $ &
$ ( -3404 , 305 , 3 ) $  ,
$ ( -4097 , 305 , 3 ) $
\\ \hline
$ 7551 $ &
$ ( -20729 , 845 , 3 ) $  ,
$ ( 13178 , 845 , 3 ) $
\\ \hline
$ 7561 $ &
$ ( -93205 , 2521 , 3 ) $  ,
$ ( 85644 , 2521 , 3 ) $
\\ \hline
$ 7579 $ &
$ ( 51209 , 1825 , 3 ) $  ,
$ ( -58788 , 1825 , 3 ) $
\\ \hline
$ 7775 $ &
$ ( -3419 , 313 , 3 ) $  ,
$ ( -4356 , 313 , 3 ) $
\\ \hline
$ 7813 $ &
$ ( 83553173 , 240805 , 3 ) $  ,
$ ( -83560986 , 240805 , 3 ) $
\\ \hline
$ 7847 $ &
$ ( 19841 , 65 , 5 ) $  ,
$ ( -27688 , 65 , 5 ) $
\\ \hline
$ 7849 $ &
$ ( 33588 , 1417 , 3 ) $  ,
$ ( -41437 , 1417 , 3 ) $
\\ \hline
$ 7857 $ &
$ ( -187 , 389 , 3 ) $  ,
$ ( -7670 , 389 , 3 ) $
\\ \hline
$ 7957 $ &
$ ( 148005 , 3589 , 3 ) $  ,
$ ( -155962 , 3589 , 3 ) $
\\ \hline
$ 7993 $ &
$ ( -101196 , 2665 , 3 ) $  ,
$ ( 93203 , 2665 , 3 ) $
\\ \hline
$ 7999 $ &
$ ( -10475 , 41 , 5 ) $  ,
$ ( 2476 , 41 , 5 ) $
\\ \hline
$ 8063 $ &
$ ( -4986 , 325 , 3 ) $  ,
$ ( -3077 , 325 , 3 ) $
\\ \hline
$ 8073 $ &
$ ( -6424 , 353 , 3 ) $  ,
$ ( -1649 , 353 , 3 ) $
\\ \hline
$ 8105 $ &
$ ( -12418 , 557 , 3 ) $  ,
$ ( 4313 , 557 , 3 ) $
\\ \hline
$ 8217 $ &
$ ( 11635 , 809 , 3 ) $  ,
$ ( -19852 , 809 , 3 ) $
\\ \hline
$ 8307 $ &
$ ( -23740 , 929 , 3 ) $  ,
$ ( 15433 , 929 , 3 ) $
\\ \hline
$ 8437 $ &
$ ( -109631 , 2813 , 3 ) $  ,
$ ( 101194 , 2813 , 3 ) $
\\ \hline
$ 8541 $ &
$ ( 15689 , 941 , 3 ) $  ,
$ ( -24230 , 941 , 3 ) $
\\ \hline
$ 8549 $ &
$ ( 413829 , 7045 , 3 ) $  ,
$ ( -422378 , 7045 , 3 ) $
\\ \hline
$ 8659 $ &
$ ( -9361 , 445 , 3 ) $  ,
$ ( 702 , 445 , 3 ) $
\\ \hline
$ 8671 $ &
$ ( 1159 , 461 , 3 ) $  ,
$ ( -9830 , 461 , 3 ) $
\\ \hline
$ 8725 $ &
$ ( 3166 , 533 , 3 ) $  ,
$ ( -11891 , 533 , 3 ) $
\\ \hline
$ 8789 $ &
$ ( -2158 , 365 , 3 ) $  ,
$ ( -6631 , 365 , 3 ) $
\\ \hline
$ 8829 $ &
$ ( -1712 , 377 , 3 ) $  ,
$ ( -7117 , 377 , 3 ) $
\\ \hline
$ 8893 $ &
$ ( -118522 , 2965 , 3 ) $  ,
$ ( 109629 , 2965 , 3 ) $
\\ \hline
$ 9017 $ &
$ ( 781300 , 10729 , 3 ) $  ,
$ ( -790317 , 10729 , 3 ) $
\\ \hline
$ 9111 $ &
$ ( -6469 , 5 , 11 ) $  ,
$ ( -2642 , 5 , 11 ) $
\\ \hline
$ 9131 $ &
$ ( 1451 , 485 , 3 ) $  ,
$ ( -10582 , 485 , 3 ) $
\\ \hline
$ 9139 $ &
$ ( -1199 , 401 , 3 ) $  ,
$ ( -7940 , 401 , 3 ) $
\\ \hline
$ 9217 $ &
$ ( -4735 , 349 , 3 ) $  ,
$ ( -4482 , 349 , 3 ) $
\\ \hline
$ 9269 $ &
$ ( 2961 , 541 , 3 ) $  ,
$ ( -12230 , 541 , 3 ) $
\\ \hline
$ 9287 $ &
$ ( 11583 , 829 , 3 ) $  ,
$ ( -20870 , 829 , 3 ) $  ,
$ ( 6084559 , 42013 , 3 ) $  ,
$ ( -6093846 , 42013 , 3 ) $
\\ \hline
$ 9361 $ &
$ ( -127881 , 3121 , 3 ) $  ,
$ ( 118520 , 3121 , 3 ) $  ,
$ ( 17082 , 997 , 3 ) $  ,
$ ( -26443 , 997 , 3 ) $
\\ \hline
$ 9603 $ &
$ ( 5267 , 629 , 3 ) $  ,
$ ( -14870 , 629 , 3 ) $
\\ \hline
$ 9647 $ &
$ ( -6461 , 373 , 3 ) $  ,
$ ( -3186 , 373 , 3 ) $
\\ \hline
$ 9703 $ &
$ ( 8684 , 745 , 3 ) $  ,
$ ( -18387 , 745 , 3 ) $
\\ \hline
$ 9729 $ &
$ ( 124306 , 3221 , 3 ) $  ,
$ ( -134035 , 3221 , 3 ) $
\\ \hline
$ 9801 $ &
$ ( -4354 , 365 , 3 ) $  ,
$ ( -5447 , 365 , 3 ) $
\\ \hline
$ 9841 $ &
$ ( -137720 , 3281 , 3 ) $  ,
$ ( 127879 , 3281 , 3 ) $
\\ \hline
$ 9855 $ &
$ ( 26962 , 1277 , 3 ) $  ,
$ ( -36817 , 1277 , 3 ) $
\\ \hline
$ 9919 $ &
$ ( -11268 , 505 , 3 ) $  ,
$ ( 1349 , 505 , 3 ) $
\\ \hline
$ 9927 $ &
$ ( -30602 , 1109 , 3 ) $  ,
$ ( 20675 , 1109 , 3 ) $
\\ \hline
$ 9999 $ &
$ ( 10705088 , 61217 , 3 ) $  ,
$ ( -10715087 , 61217 , 3 ) $
\\ \hline
%Total time: 272.269 seconds, Total memory usage: 32.09MB
\end{longtable}
\end{center}

\begin{center}
\begin{longtable}{|c|l|}
\caption{Triples of non-trivial primitive solutions $(x,|y|,n)$ 
\\
of equation~\eqref{eq:main} for $d=2, n=4$ and for $1 \leq r\leq \bound$.}
  \label{table:solutionsAP24}
 \\
\hline \multicolumn{1}{|c|}{\boldmath$r$} & \multicolumn{1}{c|}{\boldmath$(x,|y|,n)$}
\\ \hline 
\endfirsthead

\multicolumn{2}{c}
{{\bfseries \tablename\ \thetable{} -- continued from previous page}} \\
\hline \multicolumn{1}{|c|}{\boldmath$r$} &
\multicolumn{1}{c|}{\boldmath$(x,|y| ,n)$}  \\ \hline 
\endhead

\hline \multicolumn{2}{|r|}{{Continued on next page}} \\ \hline
\endfoot

\hline \hline
\endlastfoot
%Magma V2.19-10    Tue Aug 14 2018 16:30:42 on Vanditas-MacBook-Air [Seed = 
%1058443770]
%Type ? for help.  Type <Ctrl>-D to quit.
$ 1 $ &
$ ( 118 , 13 , 4 ) $  ,
$ ( -121 , 13 , 4 ) $  ,
$ ( -1 , 1 , 4 ) $  ,
$ ( -2 , 1 , 4 ) $
%$ ( -1 , \pm1 , 2j ) $  ,
%$ ( -2 , \pm1 , 2j ) $  
\\ \hline
$ 17 $ &
$ ( -10 , 5 , 4 ) $  ,
$ ( -41 , 5 , 4 ) $
\\ \hline
$ 31 $ &
$ ( -55 , 5 , 4 ) $  ,
$ ( -38 , 5 , 4 ) $
\\ \hline
$ 79 $ &
$ ( -319 , 17 , 4 ) $  ,
$ ( 82 , 17 , 4 ) $
\\ \hline
$ 191 $ &
$ ( -718 , 25 , 4 ) $  ,
$ ( 145 , 25 , 4 ) $
\\ \hline
$ 239 $ &
$ ( -359 , 13 , 4 ) $  ,
$ ( -358 , 13 , 4 ) $
\\ \hline
$ 241 $ &
$ ( 599 , 37 , 4 ) $  ,
$ ( -1322 , 37 , 4 ) $
\\ \hline
$ 401 $ &
$ ( -562 , 17 , 4 ) $  ,
$ ( -641 , 17 , 4 ) $
\\ \hline
$ 799 $ &
$ ( -79 , 41 , 4 ) $  ,
$ ( -758 , 29 , 4 ) $  ,
$ ( -1639 , 29 , 4 ) $  ,
$ ( -2318 , 41 , 4 ) $
\\ \hline
$ 863 $ &
$ ( -1199 , 25 , 4 ) $  ,
$ ( -1390 , 25 , 4 ) $
\\ \hline
$ 881 $ &
$ ( -1721 , 29 , 4 ) $  ,
$ ( -922 , 29 , 4 ) $
\\ \hline
$ 911 $ &
$ ( -6455 , 85 , 4 ) $  ,
$ ( 3722 , 85 , 4 ) $
\\ \hline
$ 1279 $ &
$ ( -38 , 53 , 4 ) $  ,
$ ( -3799 , 53 , 4 ) $
\\ \hline
$ 1361 $ &
$ ( -7601 , 89 , 4 ) $  ,
$ ( 3518 , 89 , 4 ) $
\\ \hline
$ 1457 $ &
$ ( 8839 , 125 , 4 ) $  ,
$ ( -13210 , 125 , 4 ) $
\\ \hline
$ 1649 $ &
$ ( 398 , 65 , 4 ) $  ,
$ ( -5345 , 65 , 4 ) $
\\ \hline
$ 1697 $ &
$ ( 319 , 65 , 4 ) $  ,
$ ( -5410 , 65 , 4 ) $
\\ \hline
$ 1921 $ &
$ ( -2761 , 37 , 4 ) $  ,
$ ( -3002 , 37 , 4 ) $
\\ \hline
$ 2159 $ &
$ ( -839 , 61 , 4 ) $  ,
$ ( -5638 , 61 , 4 ) $
\\ \hline
$ 2239 $ &
$ ( -2959 , 41 , 4 ) $  ,
$ ( -3758 , 41 , 4 ) $
\\ \hline
$ 2719 $ &
$ ( -19718 , 149 , 4 ) $  ,
$ ( 11561 , 149 , 4 ) $
\\ \hline
$ 3503 $ &
$ ( 24410 , 205 , 4 ) $  ,
$ ( -34919 , 205 , 4 ) $
\\ \hline
$ 3761 $ &
$ ( -5002 , 53 , 4 ) $  ,
$ ( -6281 , 53 , 4 ) $
\\ \hline
$ 4369 $ &
$ ( 43055 , 265 , 4 ) $  ,
$ ( -56162 , 265 , 4 ) $
\\ \hline
$ 4559 $ &
$ ( -9839 , 73 , 4 ) $  ,
$ ( -3838 , 73 , 4 ) $
\\ \hline
$ 4703 $ &
$ ( -2519 , 85 , 4 ) $  ,
$ ( -11590 , 85 , 4 ) $
\\ \hline
$ 4799 $ &
$ ( -8278 , 61 , 4 ) $  ,
$ ( -6119 , 61 , 4 ) $
\\ \hline
$ 5441 $ &
$ ( -14842 , 101 , 4 ) $  ,
$ ( -1481 , 101 , 4 ) $
\\ \hline
$ 5729 $ &
$ ( -9442 , 65 , 4 ) $  ,
$ ( -7745 , 65 , 4 ) $
\\ \hline
$ 5743 $ &
$ ( -9439 , 65 , 4 ) $  ,
$ ( -7790 , 65 , 4 ) $
\\ \hline
$ 6001 $ &
$ ( -11281 , 73 , 4 ) $  ,
$ ( -6722 , 73 , 4 ) $
\\ \hline
$ 6239 $ &
$ ( -1558 , 109 , 4 ) $  ,
$ ( -17159 , 109 , 4 ) $
\\ \hline
$ 7361 $ &
$ ( 9799 , 173 , 4 ) $  ,
$ ( -31882 , 173 , 4 ) $
\\ \hline
$ 7663 $ &
$ ( -35390 , 185 , 4 ) $  ,
$ ( 12401 , 185 , 4 ) $
\\ \hline
$ 7681 $ &
$ ( 42598 , 277 , 4 ) $  ,
$ ( -65641 , 277 , 4 ) $
\\ \hline
$ 8401 $ &
$ ( -17761 , 97 , 4 ) $  ,
$ ( -7442 , 97 , 4 ) $
\\ \hline
$ 8959 $ &
$ ( -113839 , 377 , 4 ) $  ,
$ ( 86962 , 377 , 4 ) $  ,
$ ( -5599 , 113 , 4 ) $  ,
$ ( -21278 , 113 , 4 ) $
\\ \hline
$ 9071 $ &
$ ( -11255 , 85 , 4 ) $  ,
$ ( -15958 , 85 , 4 ) $
\\ \hline
$ 9601 $ &
$ ( 133199 , 457 , 4 ) $  ,
$ ( -162002 , 457 , 4 ) $
\\ \hline
\end{longtable}
\end{center}

\begin{center}
\begin{longtable}{|c|l|}
\caption{Triples of non-trivial primitive solutions $(x,y,n)$ 
\\
of equation~\eqref{eq:main} for $d=6$ and prime $n\geq 3$ for $1 \leq r\leq \bound$.}
  \label{table:solutionsAP6}
 \\
\hline \multicolumn{1}{|c|}{\boldmath$r$} & \multicolumn{1}{c|}{\boldmath$(x,y,n)$}
\\ \hline 
\endfirsthead

\multicolumn{2}{c}
{{\bfseries \tablename\ \thetable{} -- continued from previous page}} \\
\hline \multicolumn{1}{|c|}{\boldmath$r$} &
\multicolumn{1}{c|}{\boldmath$(x,y,n)$}  \\ \hline 
\endhead

\hline \multicolumn{2}{|r|}{{Continued on next page}} \\ \hline
\endfoot

\hline \hline
\endlastfoot
$ 13 $ &
$ ( -20 , 19 , 3 ) $  ,
$ ( -71 , 19 , 3 ) $
\\ \hline
$ 23 $ &
$ ( -22 , 31 , 3 ) $  ,
$ ( -139 , 31 , 3 ) $
\\ \hline
$ 55 $ &
$ ( -828 , 19 , 5 ) $  ,
$ ( 443 , 19 , 5 ) $
\\ \hline
$ 347 $ &
$ ( -1525 , 139 , 3 ) $  ,
$ ( -904 , 139 , 3 ) $
\\ \hline
$ 365 $ &
$ ( 4082 , 559 , 3 ) $  ,
$ ( -6637 , 559 , 3 ) $
\\ \hline
$ 455 $ &
$ ( 1970807 , 28579 , 3 ) $  ,
$ ( -1973992 , 28579 , 3 ) $
\\ \hline
$ 527 $ &
$ ( -2554 , 199 , 3 ) $  ,
$ ( -1135 , 199 , 3 ) $
\\ \hline
$ 535 $ &
$ ( 4348 , 619 , 3 ) $  ,
$ ( -8093 , 619 , 3 ) $
\\ \hline
$ 679 $ &
$ ( 12697 , 1111 , 3 ) $  ,
$ ( -17450 , 1111 , 3 ) $
\\ \hline
$ 743 $ &
$ ( -3907 , 271 , 3 ) $  ,
$ ( -1294 , 271 , 3 ) $
\\ \hline
$ 851 $ &
$ ( 2328605 , 31951 , 3 ) $  ,
$ ( -2334562 , 31951 , 3 ) $
\\ \hline
$ 1145 $ &
$ ( -3034 , 31 , 5 ) $  ,
$ ( -4981 , 31 , 5 ) $
\\ \hline
$ 1283 $ &
$ ( -7729 , 451 , 3 ) $  ,
$ ( -1252 , 451 , 3 ) $
\\ \hline
$ 1391 $ &
$ ( 56362832 , 267139 , 3 ) $  ,
$ ( -56372569 , 267139 , 3 ) $
\\ \hline
$ 1607 $ &
$ ( -10270 , 559 , 3 ) $  ,
$ ( -979 , 559 , 3 ) $
\\ \hline
$ 1615 $ &
$ ( 1231 , 691 , 3 ) $  ,
$ ( -12536 , 691 , 3 ) $
\\ \hline
$ 1985 $ &
$ ( -4999 , 451 , 3 ) $  ,
$ ( -8896 , 451 , 3 ) $
\\ \hline
$ 2165 $ &
$ ( -6922 , 439 , 3 ) $  ,
$ ( -8233 , 439 , 3 ) $
\\ \hline
$ 2191 $ &
$ ( 5482 , 1039 , 3 ) $  ,
$ ( -20819 , 1039 , 3 ) $
\\ \hline
$ 2263 $ &
$ ( 1360645 , 22399 , 3 ) $  ,
$ ( -1376486 , 22399 , 3 ) $
\\ \hline
$ 2363 $ &
$ ( -16792 , 811 , 3 ) $  ,
$ ( 251 , 811 , 3 ) $
\\ \hline
$ 2669 $ &
$ ( 214052 , 6691 , 3 ) $  ,
$ ( -232735 , 6691 , 3 ) $
\\ \hline
$ 2813 $ &
$ ( 1109606 , 19591 , 3 ) $  ,
$ ( -1129297 , 19591 , 3 ) $
\\ \hline
$ 2893 $ &
$ ( 53803 , 2911 , 3 ) $  ,
$ ( -74054 , 2911 , 3 ) $
\\ \hline
$ 2933 $ &
$ ( 865 , 19 , 7 ) $  ,
$ ( -21396 , 19 , 7 ) $
\\ \hline
$ 2983 $ &
$ ( 302191 , 8371 , 3 ) $  ,
$ ( -323072 , 8371 , 3 ) $
\\ \hline
$ 3101 $ &
$ ( 7328 , 1291 , 3 ) $  ,
$ ( -29035 , 1291 , 3 ) $
\\ \hline
$ 3263 $ &
$ ( -25474 , 1111 , 3 ) $  ,
$ ( 2633 , 1111 , 3 ) $
\\ \hline
$ 3451 $ &
$ ( -1049 , 979 , 3 ) $  ,
$ ( -23108 , 979 , 3 ) $
\\ \hline
$ 3767 $ &
$ ( -30715 , 1279 , 3 ) $  ,
$ ( 4346 , 1279 , 3 ) $
\\ \hline
$ 4117 $ &
$ ( 263895274 , 747631 , 3 ) $  ,
$ ( -263924093 , 747631 , 3 ) $
\\ \hline
$ 4199 $ &
$ ( 90320 , 4051 , 3 ) $  ,
$ ( -119713 , 4051 , 3 ) $
\\ \hline
$ 4307 $ &
$ ( -36604 , 1459 , 3 ) $  ,
$ ( 6455 , 1459 , 3 ) $
\\ \hline
$ 4315 $ &
$ ( -7631 , 871 , 3 ) $  ,
$ ( -22574 , 871 , 3 ) $
\\ \hline
$ 4387 $ &
$ ( 3160291 , 39259 , 3 ) $  ,
$ ( -3191000 , 39259 , 3 ) $
\\ \hline
$ 4883 $ &
$ ( -43177 , 1651 , 3 ) $  ,
$ ( 8996 , 1651 , 3 ) $
\\ \hline
$ 5369 $ &
$ ( 503 , 1399 , 3 ) $  ,
$ ( -38086 , 1399 , 3 ) $
\\ \hline
$ 5423 $ &
$ ( 36224 , 2659 , 3 ) $  ,
$ ( -74185 , 2659 , 3 ) $
\\ \hline
$ 5719 $ &
$ ( -16178 , 871 , 3 ) $  ,
$ ( -23855 , 871 , 3 ) $
\\ \hline
$ 5935 $ &
$ ( -13448 , 979 , 3 ) $  ,
$ ( -28097 , 979 , 3 ) $
\\ \hline
$ 5971 $ &
$ ( -19613 , 859 , 3 ) $  ,
$ ( -22184 , 859 , 3 ) $
\\ \hline
$ 6143 $ &
$ ( -58519 , 2071 , 3 ) $  ,
$ ( 15518 , 2071 , 3 ) $
\\ \hline
$ 6827 $ &
$ ( -67360 , 2299 , 3 ) $  ,
$ ( 19571 , 2299 , 3 ) $
\\ \hline
$ 7501 $ &
$ ( 66655 , 3751 , 3 ) $  ,
$ ( -119162 , 3751 , 3 ) $
\\ \hline
$ 7547 $ &
$ ( -77029 , 2539 , 3 ) $  ,
$ ( 24200 , 2539 , 3 ) $
\\ \hline
$ 8303 $ &
$ ( -87562 , 2791 , 3 ) $  ,
$ ( 29441 , 2791 , 3 ) $
\\ \hline
$ 8987 $ &
$ ( 18857 , 2551 , 3 ) $  ,
$ ( -81766 , 2551 , 3 ) $
\\ \hline
$ 9715 $ &
$ ( -28034 , 1231 , 3 ) $  ,
$ ( -39971 , 1231 , 3 ) $
\\ \hline
$ 9923 $ &
$ ( -111364 , 3331 , 3 ) $  ,
$ ( 41903 , 3331 , 3 ) $
\\ \hline
\end{longtable}
\end{center}

%-----------------------------------------------------------------------------------------------------------

\end{document}